\documentclass[11pt,leqno]{amsart}
\usepackage{amsfonts,amssymb}
\usepackage{graphicx,tikz}
\usepackage{tikz-cd} 
\usepackage{float}
\usepackage{amsmath, amsthm, amsfonts}
\usepackage{hyperref}
\usepackage{enumitem}  
\usepackage[all,cmtip]{xy}
\usepackage[capitalise]{cleveref}
\usepackage{comment}
\usepackage{enumitem}
\usepackage{amsrefs}
\usepackage{nicefrac}
\usepackage[official]{eurosym}
\usepackage{graphicx}
\usepackage{nomencl}
\usepackage{amsfonts}
\usepackage{hyperref}
\usepackage{tabto}
\usepackage{amssymb}
\usepackage{fancyhdr}
\usepackage{amscd}
\usepackage[english]{babel}

\usepackage[margin=1.15in]{geometry}
\theoremstyle{plain}

\newtheorem*{theorem*}{Theorem}
\newtheorem{thm}{Theorem}[section]
\newtheorem{pr}[thm]{Proposition}
\newtheorem{cor}[thm]{Corollary}
\newtheorem{lem}[thm]{Lemma}
\newtheorem{theorem}{Theorem}

\theoremstyle{definition}

\makeatletter
\def\cleardoublepage{\clearpage\if@twoside \ifodd\c@page\else
	\hbox{}
	\thispagestyle{empty}
	\newpage
	\if@twocolumn\hbox{}\newpage\fi\fi\fi}
\makeatother
\DeclareMathOperator{\Aut}{Aut}
\DeclareMathOperator{\St}{St}

\DeclareMathOperator{\Rist}{Rist}

\DeclareMathOperator{\h}{hdim}

\def\ZZ{\mathbb{Z}}
\def\N{\mathbb{N}}

\numberwithin{equation}{section}


\begin{document}
	
	\title[$p$-Basilica groups]{$p$-Basilica groups}

\author[E. Di Domenico]{Elena Di Domenico}
\address{Elena Di Domenico: Dipartimento di Matematica, Universit\`a degli Studi di
Trento, 38123 Trento  Italy; Department of Mathematics, University of the Basque Country UPV/EHU, 48080 Bilbao, Spain}
 \email{elena.didomenico@unitn.it}
 
\author[G.\,A. Fern\'{a}ndez-Alcober]{Gustavo\,A. Fern\'{a}ndez-Alcober}
 
 \address{Gustavo\,A. Fern\'{a}ndez-Alcober: Department of Mathematics, University of the Basque Country UPV/EHU, 48080 Bilbao, Spain}
 \email{gustavo.fernandez@ehu.eus}

\author[M. Noce]{Marialaura Noce}
\address{Marialaura Noce: Mathematisches Institut, Georg-August Universit\"{a}t zu G\"{o}ttingen, Bunsenstr. 3-5, 37073 G\"{o}ttingen, Germany}
\email{mnoce@unisa.it}

 \author[A. Thillaisundaram]{Anitha Thillaisundaram}
 
 \address{Anitha Thillaisundaram: School of Mathematics and Physics, University of Lincoln,
 	Brayford Pool, Lincoln LN6 7TS, United Kingdom}
 \email{anitha.t@cantab.net}

\thanks{The first three authors are supported by the Spanish Government grant MTM2017-86802-P, partly with FEDER funds, and by the Basque Government grant IT974-16.
The first and third authors are also partially supported by the National Group for Algebraic and Geometric Structures, and their Applications (GNSAGA -- INdAM). The first author  acknowledges support from the Department of Mathematics of the University of Trento. The third author acknowledges financial support from a London Mathematical Society Joint Research Groups in the UK (Scheme 3) grant. The fourth author acknowledges support from EPSRC, grant EP/T005068/1.}

\begin{abstract}
We consider a generalisation of the Basilica group to all odd primes: the $p$-Basilica groups acting on the $p$-adic tree. We show that the $p$-Basilica groups have the $p$-congruence subgroup property but not the congruence subgroup property nor the weak congruence subgroup property. This provides the first examples of weakly branch groups with such properties.
In addition, the $p$-Basilica groups give the first examples of weakly branch, but not branch, groups which are super strongly fractal. 
We compute the orders of the congruence quotients of these groups, which enable us to determine the Hausdorff dimensions of the $p$-Basilica groups.
Lastly, we show that the $p$-Basilica groups do not possess maximal subgroups of infinite
index and that they have infinitely many non-normal maximal subgroups.
\end{abstract}

\maketitle

\section{Introduction}

Let $p$ be a prime and let $T$ be the $p$-adic tree. Groups acting on $p$-adic trees have been well studied over the past decades, owing to their nice structure, their importance in the theory of just infinite groups, and the fact that many such groups have exotic properties; see~\cite{Handbook} for a good introduction. Lots of the interesting examples of such groups share one common property: that of being \emph{branch} or \emph{weakly branch}, where branchness is a measure of how close the structure of the group resembles the structure of the full automorphism group of the tree~$T$; see Section~2 for precise definitions.

An example of such a group is the Basilica group. This group acts on the binary tree, is weakly branch but not branch, is torsion-free, is of exponential growth, and is not subexponentially amenable~\cite{Zuk}. The Basilica group is generated by two elements, $a$ and $b$, which are recursively defined as follows:
\[
a=(1,b) \quad\text{and}\quad b=(1,a)\sigma
\]
where $\sigma$ is the cyclic permutation $(1\,2)$, which swaps the two maximal subtrees, and the notation $(x,y)$ indicates the independent actions on the respective maximal subtrees, for $x$ and $y$ automorphisms of the binary tree.

In this paper, we are interested in a natural generalisation of the Basilica group, which we call the \emph{$p$-Basilica group}, that acts on the $p$-adic tree, for $p$ any prime. Such a group  $G$ is generated by the following $2$ elements:
\[
a=(1,\overset{p-1}\ldots,1,b) \quad\text{and}\quad b=(1,\overset{p-1}\ldots,1,a)\sigma
\]
where $\sigma$ is the cyclic permutation $(1\,2\,\cdots \, p)$. Clearly the 2-Basilica group coincides with the Basilica group. This generalisation of the Basilica group mirrors Sidki and Silva's generalisation of the Brunner-Sidki-Vieira group; see~\cite{SidkiSilva} and~\cite{BSV}. A different generalisation of the Basilica group to the $p$-adic tree, with $p$ generators, was first investigated by Sasse in her Master thesis~\cite{Sasse}, and Sasse's work has been recently developed further by Petschick and Rajeev~\cite{PR}. As seen below, our $2$-generator $p$-Basilica groups,  also known in~\cite{PR} as Basilica groups of level~$2$, are more similar to the Basilica group. The generalisations of the Basilica groups considered by Sasse, Petschick and Rajeev include the Basilica groups of levels strictly greater than 2, and they differ more significantly from the Basilica group.

We prove the following in Sections~\ref{sec:first-properties} and~\ref{sec:commutator}; see Theorem~\ref{thm:G/G'}, Theorem~\ref{thm:weakly-branch} and Lemma~\ref{thm:key}.

\begin{theorem}\label{thm:structural-results}
Let $G$ be a $p$-Basilica group, for $p$ a prime. Then $G$ is not branch, but it is weakly regular branch over~$G'$.
Furthermore:
\begin{enumerate}
    \item $G/G'\cong \mathbb{Z} \times \mathbb{Z}$.
    \item $G'/\gamma_3(G)\cong \mathbb{Z}$.
    \item  $G'/G''\cong \mathbb{Z}^{2p-1}$.
    \item  $\gamma_3(G)/G''\cong \mathbb{Z}^{2p-2}$.
\end{enumerate}
\end{theorem}

Additionally in Sections~\ref{sec:first-properties} and~\ref{sec:commutator} we establish other basic properties of the $p$-Basilica groups~$G$, such as being  torsion-free (Theorem~\ref{thm:torsion-free}), contracting (Theorem~\ref{thm: contracting}), just non-solvable (Corollary~\ref{cor: Just non solvable}), and having its automorphism group $\Aut(G)$ equal the normaliser of~$G$ in $\Aut(T)$ (Corollary~\ref{cor:Aut(G)}).
We also show that the groups are super strongly fractal (Theorem~\ref{thm:super-strongly-fractal}), which means for any $n\in\mathbb{N}$, the projection of the $n$th level stabiliser $\St_G(n)$ at any $n$th level vertex is the whole of~$G$; see Section~2 for the precise definition. This yields the first examples of finitely generated weakly branch, but not branch, groups that are super strongly fractal.

Now, one of the main properties concerning the $p$-Basilica groups that we investigate is the congruence  subgroup property, where we say that $\mathcal{G}\le \Aut(T)$ has the \emph{congruence subgroup property} if every finite-index subgroup of~$\mathcal{G}$ contains a level stabiliser
$\St_{\mathcal{G}}(n)$ for some $n\in\mathbb{N}$.
Equivalently, the group $\mathcal{G}$ has the congruence subgroup property if the profinite completion of~$\mathcal{G}$ equals its closure in $\text{Aut}(T)$.

Garrido and Uria-Albizuri~\cite{pcongruence} introduced a weaker version of the congruence subgroup property: a group $\mathcal{G}\le \Aut(T)$ is said to have the \emph{$p$-congruence subgroup property} if every normal subgroup of $p$-power index contains some level stabiliser.
In~\cite{pcongruence}, examples of weakly branch, but not branch, groups with the $p$-congruence subgroup property and not the congruence subgroup property were provided.
For $p$ odd, their examples were the Grigorchuk-Gupta-Sidki (GGS-)groups defined by the constant vector, and for $p=2$, their example was the Basilica group.
In Section~\ref{sec:CSP}, we extend this result to $p$-Basilica groups, for all odd primes $p$:

\begin{theorem}
\label{thm:CSP}
Let $G$ be a $p$-Basilica group, for $p$ a prime. Then $G$ has the $p$-congruence subgroup property but not the congruence subgroup property nor the weak congruence subgroup property.
\end{theorem}

\noindent
We recall that a group $\mathcal{G}\le \Aut(T)$ has the \emph{weak congruence subgroup property} if every finite-index subgroup contains the derived subgroup of some level stabiliser; cf.~\cite{Segal}.
The $p$-Basilica groups are the first examples of  
weakly regular branch groups with the $p$-congruence subgroup property but not the weak congruence subgroup property.

In Subsection~\ref{sec:Hausdorff}, we  compute the orders of the congruence quotients $G/\St_G(n)$ for all $n\in\mathbb{N}$, for a $p$-Basilica group~$G$. This enables us to compute the Hausdorff dimension of the closure of the $p$-Basilica group~$G$  in the group~$\Gamma$ of $p$-adic automorphisms of~$T$. We recall that 
\[
\Gamma \cong \varprojlim_{n\in\mathbb{N}} C_{p} \wr \overset{n}\cdots \wr C_{p}
\]
is a Sylow pro-$p$ subgroup of $\Aut(T)$ corresponding to the $p$-cycle $(1\,2\, \cdots \,p)$.
For a subgroup~$\mathcal{G}$ of~$\Gamma$, the Hausdorff dimension of the closure of~$\mathcal{G}$ in~$\Gamma$ is given by
\begin{equation}
\label{eqn:hausdorff dim}
\h_\Gamma (\overline{\mathcal{G}})
=
\varliminf_{n\to \infty}
\frac{\log|\mathcal{G}:\St_{\mathcal{G}}(n) |}{\log| \Gamma :\St_\Gamma(n)| } \in [0,1], 
\end{equation}
where $\varliminf$ represents the lower limit.
The Hausdorff dimension of $\overline{\mathcal{G}}$ is a measure of how dense
$\overline{\mathcal{G}}$ is in $\Gamma$.
This concept was first applied by Abercrombie~\cite{Abercrombie} and by Barnea and Shalev~\cite{BaSh97} in the more general setting of profinite groups. We note that the Hausdorff dimension of the closures of several prominent weakly branch groups, such as the first \cite{NewHorizons} and second \cite{MarAn} Grigorchuk groups, the siblings of the first Grigorchuk group \cite{Sunic}, the GGS-groups \cite{FAZR}, the branch path groups \cite{FAGT}, and generalisations of the Hanoi tower groups \cite{HanSki}, have been computed.

\begin{theorem}\label{thm:hdim}
Let $G$ be a $p$-Basilica group, for $p$ a prime.
Then:
\begin{enumerate}
    \item The orders of the congruence quotients of $G$ are given by
    \[
    \log_p|G:\St_{G}(n)| =
    \begin{cases}
    p^{n-1}+p^{n-3}+\cdots+ p^3+p+\frac{n}{2} & \text{for }n \text{ even},\\
p^{n-1}+ p^{n-3}+\cdots+ p^4+p^2+\frac{n+1}{2} & \text{for }n \text{ odd}.
    \end{cases}
    \]
    \item  The Hausdorff dimension of the closure of $G$ in $\Gamma$ is
    \[
    \h_\Gamma(\overline{G})= \dfrac{p}{p+1}.
    \]
\end{enumerate}
\end{theorem}

\medskip

In Section~\ref{sec:growth}, we give a recursive presentation, a so-called $L$-presentation, for the $p$-Basilica groups (Proposition~\ref{pr:11}), plus we  show that the $p$-Basilica groups are amenable but not elementary subexponentially amenable (Lemma~\ref{lem:amenable}), and have exponential growth (Theorem~\ref{thm:exponentialgrowth}); we refer to Section~\ref{sec:growth} for the definitions. To the best of our knowledge, the only other infinite family of weakly branch groups that are amenable but not elementary subexponentially amenable is the family of $p$-generator Basilica groups acting on the $p$-adic tree; see~\cite{Sasse}. 

Francoeur~\cite[Thm.~4.28]{Francoeur-paper} proved that the  Basilica group does not possess maximal subgroups of infinite index, thus providing the first example of a weakly branch but not branch group without maximal subgroups of infinite index. Also, the Basilica group has non-normal maximal subgroups \cite[Cor. 8.3.2]{Francoeur}. In Subsection~\ref{sec:nilpotent+maximal}, we extend these results to $p$-Basilica groups for all primes $p$, likewise giving another infinite family of  weakly branch groups with such properties.
Note that the first  infinite family of weakly branch, but not branch, groups without maximal subgroups of infinite index was given by Francoeur and Thillaisundaram in~\cite{FT}, namely the GGS-groups defined by the constant vector. 

\begin{theorem}\label{thm:maximal}
Let $G$ be a $p$-Basilica group, for $p$ a prime. Then all maximal subgroups of $G$ have finite index, and $G$ has infinitely many non-normal maximal subgroups.
\end{theorem}

\smallskip

\noindent
\textit{Notation.}
Throughout, we  use left-normed commutators, for example, $[x,y,z] = [[x,y],z]$.
For a group~$\mathcal{G}$, a subgroup~$H\le \mathcal{G}$ and  $g\in \mathcal{G}$, we write
$[H,g]=\langle [h,g]\mid h\in H\rangle$.
Also if $N\trianglelefteq \mathcal{G}$ then we write $g\equiv_N h$ to mean that the images of $g$ and $h$ in $\mathcal{G}/N$ coincide.
For $\mathcal{G}$ a group and $p$ a prime, we write $W_p(\mathcal{G})$ for the wreath product
of~$\mathcal{G}$ with a cyclic group of order $p$.

\smallskip

\textbf{Acknowledgements.}
We thank D.~Francoeur,  M.~Petschick and K.~Rajeev for helpful discussions.
Furthermore we are grateful to B. Klopsch for his useful comments and for pointing out Sasse's work.


\section{Preliminaries}

\subsection{The group $\Aut(T)$}

Let $p$ be a prime and let $T$ be the \textit{$p$-adic tree}, i.e.\ the rooted tree having $p$ descendants at every vertex.
If we choose an alphabet $X$ with $p$ letters, $T$ can be represented as the graph whose vertices are the elements of the free monoid $X^*$, the root is the empty word $\varnothing$, and $w$ is a descendant of $u$ provided that $w=ux$ with $x\in X$.

For a given vertex $u$, the set of vertices $uv$ with $v\in X^*$ are said to succeed $u$.
They form a tree $T_u$ rooted at $u$, which is isomorphic to $T$.
We denote by $|u|$ the length of $u$ as a word.
For every $n\in\N\cup\{0\}$, the set $L_n$ of all words of length $n$ is called the
\textit{$n$th layer} of the tree.

Automorphisms of $T$ as a graph form a group $\Aut(T)$ under composition.
Let $u$ be a vertex of $T$ and let $f\in\Aut(T)$.
We use exponential notation for images of automorphisms and, more generally, permutations.
Thus the image of $u$ under $f\in\Aut(T)$ is $u^f$.
The \textit{label} of $f$ at $u$ is the permutation $f(u)$ of the alphabet $X$ defined by the rule
\[
(ux)^f = u^f x^{f(u)},
\qquad
\text{for every $x\in X$.}
\]
The \textit{portrait} of $f$ is the set of all labels of $f$, and there is a one-to-one correspondence between automorphisms of~$T$ and portraits.
The \textit{support} of $f$ is the set of vertices with non-trivial label.
We say that $f$ is \textit{rooted} if the support is contained in the root, and $f$ is \textit{directed} if the support is infinite and consists only of descendants of a given infinite path starting at the root.

In a similar way, the section $f_u$ of $f$ at $u$ is the automorphism of $T$ defined by
\[
(uv)^f = u^f v^{f_u},
\qquad
\text{for every $v\in X^*$.}
\]
For all $f,g\in \Aut(T)$ and $u,v\in X^*$, we have $(f_u)_v=f_{uv}$, $(fg)_u=f_u g_{u^f}$,
\begin{equation}
\label{eqn:section conjugate}
(f^g)_{u^g} = (g_u)^{-1} \, f_ u \, g_{u^f}.
\end{equation}

\subsection{Subgroups of $\Aut(T)$}

For a vertex $u$ of $T$, the \textit{vertex stabiliser} $\St(u)$ is the subgroup consisting of all automorphisms of $T$ fixing $u$.
The map $\psi_u:f\mapsto f_u$ is a homomorphism from $\St(u)$ onto $\Aut(T)$.
For every $n\in\N$, the \textit{$n$th level stabiliser} is
\[
\St(n) = \bigcap_{u\in L_n} \, \St(u).
\]
Then $\St(n)$ is a normal subgroup of $\Aut(T)$ and $\Aut(T)$ is isomorphic to the inverse limit of the finite groups $\Aut(T)/\St(n)$.
Hence $\Aut(T)$ is a profinite group with $\{\St(n)\}_{n\in\N}$ as a basis of neighbourhoods of the identity.
For every $n\in\N$, we have an isomorphism
\[
\begin{matrix}
\psi_n & \colon & \St(n) & \longrightarrow & \Aut(T) \times \overset{p^n}{\cdots} \times \Aut(T)
\\[3pt]
& & f & \longmapsto & (f_u)_{u\in L_n}.
\end{matrix}
\]
The map $\psi_1$ can be extended to an isomorphism
\[
\begin{matrix}
\psi & \colon & \Aut(T) & \longrightarrow & \Aut(T) \wr S_X
\\[3pt]
& & f & \longmapsto & \psi_1(f) \, \tau,
\end{matrix}
\]
where $\tau$ is the label of $f$ at the root.
Thus we can define automorphisms of $\Aut(T)$ by giving their image under $\psi$.
If $\sigma\in S_X$ is a fixed $p$-cycle, we can obtain a Sylow pro-$p$ subgroup $\Gamma(\sigma)$ of $\Aut(T)$ by considering all automorphisms whose portrait only contains labels from
$\langle \sigma \rangle$.

Now let $G$ be a subgroup of $\Aut(T)$.
We write $\St_G(u)=\St(u)\cap G$ and $\St_G(n)=\St(n)\cap G$.
The latter is a normal subgroup of $G$ and we set $G_n=G/\St_G(n)$, which is called the
\textit{$n$th congruence quotient} of $G$.
We say that $G$ is:
\begin{itemize}
\item
\textit{Level-transitive} if it acts transitively on $L_n$ for every $n\in\N$.
\item
\textit{Self-similar} if all sections of elements of $G$ at all vertices belong to $G$.
\item
\textit{Fractal} if it is level-transitive and $\psi_u(\St_G(u))=G$ for every vertex $u$ of the tree.
\item
\textit{Super strongly fractal} if it is level-transitive and $\psi_u(\St_G(n))=G$ for every $u\in L_n$ and every $n\in\N$.
\end{itemize}
Note that if $G$ is self-similar then $\psi_n(\St_G(n))\subseteq G\times \overset{p^n}{\cdots} \times G$ for all $n$, and if furthermore $G$ is contained in a Sylow pro-$p$ subgroup $\Gamma(\sigma)$ then
$\psi(G)\subseteq G\wr \langle \sigma \rangle=W_p(G)$.
The $n$th congruence quotient of $\Gamma(\sigma)$ is isomorphic to the iterated wreath product of $n$ copies of $C_p$ and $\Gamma(\sigma)$ is the inverse limit of these finite $p$-groups.

The \textit{rigid vertex stabiliser} $\Rist_G(u)$ of a vertex $u$ in $G$ is the subgroup consisting of all automorphisms in $G$ that fix all vertices outside $T_u$.
Then for every $n\in\N$ we define the \textit{rigid $n$th level stabiliser} as:
\[
\Rist_G(n) = \langle \Rist_G(u) \mid u\in L_n \rangle =  \prod\nolimits_{u\in L_n} \, \Rist_G(u) \trianglelefteq G.
\]
If $G$ is level-transitive we say that $G$ is a \textit{branch group} if $|G:\Rist_G(n)|<\infty$ for all $n\in\N$, and that it is \textit{weakly branch} if $\Rist_G(n)\ne 1$ for all $n$.
If $G$ is level-transitive and self-similar, and $K\times \cdots \times K \subseteq \psi(K)$ for some
$K\le G$, we say that $G$ is \emph{regular branch over $K$} if $|G:K|<\infty$ and that $G$ is \textit{weakly regular branch over $K$} if $K\ne 1$.
Observe that being (weakly) regular branch implies being (weakly) branch.

\subsection{$p$-Basilica groups}

Let $X=\{x_1,\ldots,x_p\}$ and let $\Gamma$ be the Sylow pro-$p$ subgroup of $\Aut(T)$ corresponding to the $p$-cycle $\sigma=(x_1\,x_2\,\cdots\,x_p)$.
The \emph{$p$-Basilica group} is the subgroup~$G$ of~$\Gamma$ generated by the automorphisms $a$ and $b$ given by
\[
\psi(a) = (1, \dots,1, b)
\quad
\text{and}
\quad
\psi(b) = (1, \dots, 1, a) \, \sigma.
\]
Note that the $2$-Basilica group is the well-known Basilica group mentioned in the introduction.
The portraits of $a$ and $b$ are as follows.
\begin{figure}[h!]
    \centering
    \includegraphics[scale=0.8]{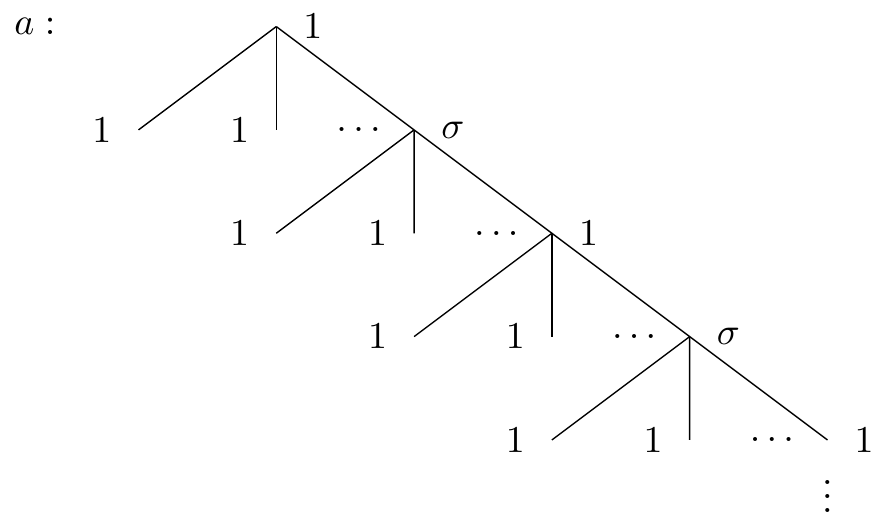}
     \includegraphics[scale=0.8]{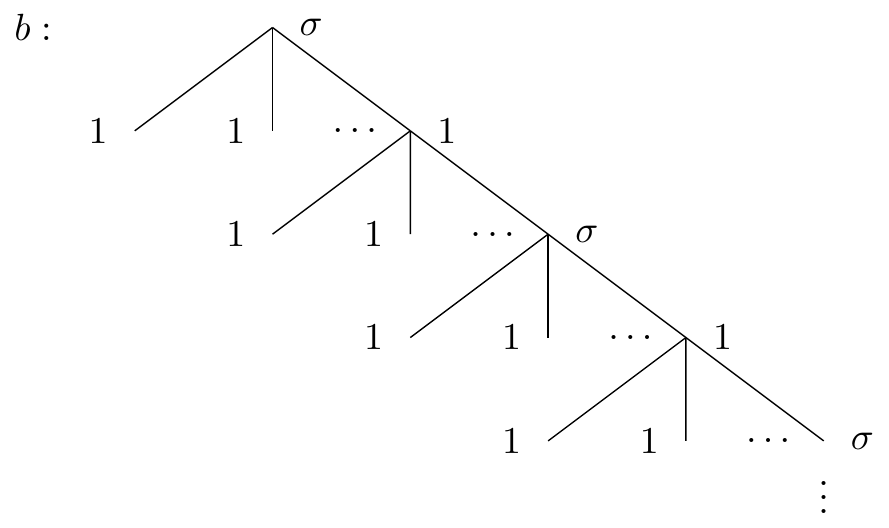}
    \caption{The portrait of the generators of a $p$-Basilica group.}
    \label{fig:my_label}
\end{figure}

\section{First properties}
\label{sec:first-properties}

In this section we prove some basic properties of the $p$-Basilica groups.
We start with the following elementary but essential result.

\begin{lem}
\label{lem:basilicaleveltransitive}
Let $G$ be a $p$-Basilica group, for a prime $p$.
Then $G$ is fractal and level-transitive. 
\end{lem}

\begin{proof}
By \cite[Lem.~2.7]{Jone}, it suffices to show that $G$ acts transitively on the first layer and that
$\psi_x(\St_G(x))=G$ for some $x\in X$ (see also \cite[Sec.~3]{Gri11}).
This is straightforward since $b$ acts transitively on the first
layer and since $\psi(a)=(1,\dots, 1, b)$ and $\psi(b^p)=(a,\dots,a)$.
\end{proof}

Next we consider the stabilisers in~$G$ of the first two levels.
Recall that $G_n=G/\St_G(n)$, and for convenience, we set $A=\langle a \rangle^G$ and
$B=\langle b \rangle^G$.

\begin{lem}
\label{lem: stabiliser}
Let $G$ be a $p$-Basilica group, for a prime $p$.
Then:
\begin{enumerate}
\item
$\St_G(1)=A\langle b^{p}\rangle=\langle  a, a^b, \ldots, a^{b^{p-1}}, b^p\rangle$
and $G_1=\langle b\St_G(1)\rangle\cong C_p$.
\item
$G_2\cong C_p\wr C_p$ is a $p$-group of maximal class of order $p^{p+1}$.
\end{enumerate}
\end{lem}

\begin{proof}
(i)
Observe that $a\in\St_G(1)$, and that $b^n\in\St_G(1)$ if and only if $p\mid n$. 
Since $\St_G(1)\trianglelefteq G$ we get $A\langle b^p\rangle\le \St_G(1)$.
This inclusion is an equality, since $G/A\langle b^p\rangle$ has order $p$ and $\St_G(1)$ is a proper subgroup of $G$.
Now the result follows by observing that
$\langle  a, a^b, \ldots, a^{b^{p-1}}, b^p\rangle \trianglelefteq G$.

(ii)
Since $\psi(b^p)=(a,\ldots,a)$ and $a\in\St_G(1)$, we have $b^p\in\St_G(2)$.
By (i), if we use the bar notation in $G_2$, we have
\[
\St_{G_2}(1)=\overline{\langle a, a^b, \ldots, a^{b^{p-1}} \rangle}.
\]
Since $\overline a$ has order $p$ in $G_2$ and the tuples
\[
\psi(a)=(1,\ldots,1,b)
\quad
\text{and}
\quad
\psi(a^{b^i})=(1,\overset{i-1}\ldots,1,a^{-1}ba,1,\ldots,1), \ \text{for } 1\le i\le p-1,
\]
commute with each other, $\St_{G_2}(1)$ is elementary abelian of order $p^p$.
Thus $|G_2|=p^{p+1}$.
To complete the proof, observe that since $G\le \Gamma$, the quotient $G_2=G/\St_G(2)$ embeds in
$\Gamma_2=\Gamma/\St_{\Gamma}(2)\cong C_p \wr C_p$, and that the latter is a $p$-group of maximal class of order $p^{p+1}$.
\end{proof}

Our next goal is to study the abelianisation of $G$.
In the remainder, let $A=\langle a \rangle^G$ and $B=\langle b \rangle^G$ as before.
Since $G=\langle a,b \rangle$, we have $A=\langle a \rangle G'$, $B=\langle b \rangle G'$, and
$G=AB$.
Observe that $\psi(a)=(1,\ldots,1,b)$ and $G$ being self-similar imply 
\begin{equation}
\label{psiA}
\psi(A) \subseteq B\times \cdots \times B.
\end{equation}
On the other hand, the map $\pi:W_p(G) \rightarrow G/G'$ sending
$(g_1,\ldots,g_p)\sigma^i$ to $g_1\cdots g_p G'$ is clearly a group homomorphism.
Since $\psi(b)=(1,\ldots,1,a)\sigma$, it follows that
\begin{equation}
\label{psiB}
(\pi\circ\psi)(B)\subseteq A/G'.
\end{equation}

\begin{thm}
\label{thm:G/G'}
Let $G$ be a $p$-Basilica group, for a prime $p$.
Then:
\begin{enumerate}
\item
$G/A=\langle bA \rangle$ and $G/B=\langle aB \rangle$ are infinite cyclic.
In particular, the elements $a$ and $b$ have infinite order in $G$.
\item
$G/G' = \langle aG'\rangle \times \langle bG'\rangle \cong \mathbb{Z}\times \mathbb{Z}$.
\item
$A\cap B=G'$.
\end{enumerate}
\end{thm}

\begin{proof}
Since $G/G'=A/G'\cdot B/G'$, with $A/G'=\langle aG' \rangle$ and $B/G'=\langle bG' \rangle$, both (ii) and (iii) follow immediately from (i).

We prove that $G/A$ and $G/B$ are infinite simultaneously.
Assume for a contradiction that, for some $n\in\N$, we have either $a^n\in B$ or $b^n\in A$, and let us
choose $n$ as small as possible.
If $b^n\in A\subseteq \St_G(1)$ then $n=pm$ for some $m$, and consequently
\[
\psi(b^n)=(a^m,\ldots,a^m)\in \psi(A) \subseteq B\times \cdots \times B,
\]
by (\ref{psiA}).
Hence $a^m\in B$, which is impossible since $m<n$.
On the other hand, if $a^n\in B$ then
\[
b^nG' = (\pi\circ\psi)(a^n) \in (\pi\circ\psi)(B) \subseteq A/G'
\]
by (\ref{psiB}).
Thus $b^n\in A$, which we just proved is not the case.
This completes the proof.
\end{proof}

Next we study rigid stabilisers and the branch structure of $G$.
To this purpose, the following result is very useful.
It is given in \cite[Prop.\ 2.18]{FAZR} for GGS-groups, but the same proof works more generally for level-transitive fractal groups.
We state this general version here for the convenience of the reader.

\begin{lem}
\label{lem:getting-direct-product}
Let $\mathcal{G}$ be a level-transitive fractal subgroup of $\Aut(T)$, and let $L$ and $N$ be two normal subgroups of $\mathcal{G}$.
Suppose that $L=\langle S \rangle^{\mathcal{G}}$ and that $(1,\ldots,1,s,1,\ldots,1)\in\psi(N)$ for every
$s\in S$, where $s$ appears always at the same position in the tuple.
Then $L\times \cdots \times L \subseteq \psi(N)$.
\end{lem}

\begin{thm}
\label{thm:weakly-branch}
Let $G$ be a $p$-Basilica group, for a prime $p$.
Then:
\begin{enumerate}
\item
$\Rist_G(1)=A$ with $\psi(A)=B\times \cdots \times B$.
In particular, the group~$G$ is not branch.
\item
$\psi(\St_G(1)')=G'\times \cdots \times G'$.
As a consequence, the group~$G$ is weakly regular branch over~$G'$. 
\end{enumerate}
\end{thm}

\begin{proof}
(i)
We already know from (\ref{psiA}) that $\psi(A)\subseteq B\times \cdots \times B$, and the reverse inclusion follows from Lemma \ref{lem:getting-direct-product}, since $\psi(a)=(1,\ldots,1,b)$.
Hence $A\le \Rist_G(1)\le \St_G(1)=A\langle b^p \rangle$ and so
$\Rist_G(1)=A\langle b^{pn} \rangle$ for some $n$.
Since 
\[
\psi(A\langle b^{pn} \rangle)=(B\times \cdots \times B)\langle (a^n,\ldots,a^n) \rangle
\]
and $a$ has infinite order modulo $B$ by Theorem~\ref{thm:G/G'}(i), it follows from the definition of rigid stabiliser that $n=0$.
Hence $\Rist_G(1)=A$, which has infinite index in $A$, so $G$ is not branch.

(ii)
The inclusion $\subseteq$ is clear.
For the reverse inclusion, observe that $\psi([b^{p},a])=(1,\ldots,1,[a,b])$.
Since $G'=\langle [a,b]\rangle^{G}$, the result follows from Lemma \ref{lem:getting-direct-product}.
\end{proof}

\begin{thm}
\label{thm:torsion-free}
Let $\mathcal{G}$ be a self-similar subgroup of $\Aut(T)$ and suppose that there exists a torsion-free quotient $\mathcal{G}/N$ with $N\le \St_{\mathcal{G}}(1)$.
Then $\mathcal{G}$ is torsion-free.
In particular, the $p$-Basilica groups are torsion-free.
\end{thm}

\begin{proof}
For every $n\in\N\cup\{0\}$, let $S_n$ stand for the set of torsion elements in
$\St_{\mathcal{G}}(n)\backslash \St_{\mathcal{G}}(n+1)$.
Then our goal is to prove that these sets are all empty.
By way of contradiction, suppose that $S_n\ne\varnothing$ for some $n$, which we choose as small as possible.

Since $\mathcal{G}/N$ is torsion-free and $N\le \St_{\mathcal{G}}(1)$, it is clear that $n\ge 1$.
Let $g\in S_n$ be a torsion element.
If $\psi(g)=(g_1,\ldots,g_p)$ then some $g_i$ belongs to
$\St_{\mathcal{G}}(n-1)\backslash \St_{\mathcal{G}}(n)$ and, of course, the element~$g_i$ is of finite order.
So $g_i\in S_{n-1}$, which is a contradiction.

The case of the $p$-Basilica groups follows from Theorem \ref{thm:G/G'}, since $G'\le \St_G(1)$.
\end{proof}
 
Next we prove that the $p$-Basilica groups are contracting.
Let us briefly recall this concept.
If~$\mathcal{G}$ is an arbitrary group generated by a symmetric subset~$S$
(i.e.\ a set for which $S = S^{-1}$), then for every $g \in \mathcal{G}$, 
\[
|g| = \min\{n \geq 0 \mid g = s_{1} \cdots s_{n}, \text{ for }s_{1}, \dots, s_{n} \in S\}
\]
is called the \textit{length} of~$g$ with respect to $S$.
Now assume that $\mathcal{G}$ is a self-similar subgroup of $\Aut(T)$.
Then for every $g\in \mathcal{G}$ and every $n\in\N\cup\{0\}$ we define
\[
\ell_n(g) = \max \{ |g_u| \mid u\in L_n \}.
\]
From the rule $(gh)_u=g_u h_{u^g}$ we get that the function $\ell_n$ is subadditive, i.e.\ that
\begin{equation}
\label{eqn:subadditive}
\ell_n(gh)\le \ell_n(g)+\ell_n(h)
\quad
\text{for every $g,h\in\mathcal{G}$,}
\end{equation}
and from $(g^{-1})_u=(g_{u^{g^{-1}}})^{-1}$, that
\[
\ell_n(g^{-1}) = \ell_n(g)
\quad
\text{for every $g\in\mathcal{G}$.}
\]
If there exist $\lambda<1$ and $C,L\in\mathbb{N}$ such that
\[
\ell_n(g) \le \lambda |g|+C,
\quad
\text{for every $n>L$ and every $g\in \mathcal{G}$,}
\]
then we say that the group $\mathcal{G}$ is \emph{contracting} with respect to $S$.

The following lemma is straightforward, but very useful in proving that a subgroup of $\Aut(T)$ is
contracting.

\begin{lem}
\label{lem:reduction contracting}
Let $\mathcal{G}=\langle S \rangle$ be a self-similar subgroup of $\Aut(T)$, where $S$ is symmetric
and suppose that $\ell_1(s)\le 1$ for all $s\in S$.
Then $\ell_n(g) \le \ell_{n-1}(g)$ for every $n\in\N$ and $g\in\mathcal{G}$.
\end{lem}

\begin{proof}
Observe that the condition $\ell_1(s)\le 1$ for all $s\in S$, together with
(\ref{eqn:subadditive}), imply that $\ell_1(g)\le |g|$ for every $g\in\mathcal{G}$.
Now let $u\in L_n$ and write $u=vx$ with $v\in L_{n-1}$ and $x\in X$.
Then for every $g\in\mathcal{G}$ we have
\[
|g_u| = |(g_v)_x| \le \ell_1(g_v) \le |g_v| \le \ell_{n-1}(g),
\]
and the result follows.
\end{proof}

Observe that the condition $\ell_1(s)\le 1$ for all $s\in S$ is satisfied by the set of generators
$S=\{a,b,a^{-1},b^{-1}\}$ in a $p$-Basilica group.

\begin{thm}
\label{thm: contracting}
For $p$ a prime, the $p$-Basilica group $G$ is contracting with respect to the set of generators
$S=\{a,b,a^{-1},b^{-1}\}$, with $\lambda=\frac{2}{3}$ and $C=L=1$.
\end{thm}

\begin{proof}
By Lemma \ref{lem:reduction contracting}, it suffices to prove that
\begin{equation}
\label{contracting level 2}
\ell_2(g)\le \frac{2}{3}|g|+1, \quad \text{for every $g\in G$.}
\end{equation}
Assume that we know that $\ell_2(h)\le 2$ for every $h\in G$ of length $3$.
Then we write $|g|=3j+k$ with $k\in\{0,1,2\}$ and
$g=h_1\cdots h_j f$ with $|h_1|=\cdots=|h_j|=3$ and $|f|=k$, and (\ref{contracting level 2}) immediately follows from the subadditivity of $\ell_2$.

Let us then consider an arbitrary element $h\in G$ of length $3$ and prove that $\ell_2(h)\le 2$.
An easy calculation shows that $\ell_1(g)\le 1$ for $g\in\{b^2,ba,ba^{-1}\}$ and, as a consequence,
$\ell_2(g)\le 1$ for every $g$ of length $2$ other than $a^{-1}b$ and $b^{-1}a$.
Now one of the elements $h$ or $h^{-1}$, call it $h^*$, does not contain $a^{-1}b$ or $b^{-1}a$ as a prefix.
If we write $h^*=gs$ with $g$ of length $2$, then
\[
\ell_2(h) = \ell_2(h^*) \le \ell_2(g) + \ell_2(s) \le 2,
\]
which completes the proof.
\end{proof}

\begin{cor}\label{co:wod-problem}
Let $G$ be a $p$-Basilica group, for a prime $p$.
Then $G$ has solvable word problem.
\end{cor}

\begin{proof}
This follows from Theorem~\ref{thm: contracting} and \cite[Prop.~2.13.8]{Nekrashevych}.
\end{proof}

\section{Commutator subgroup structure}\label{sec:commutator}

In Section \ref{sec:first-properties} we determined the abelianisation of a $p$-Basilica group $G$ and proved that $G$ is weakly regular branch over $G'$.
Now we study further properties of $G'$ and of other subgroups obtained by taking commutators, and obtain some consequences.

In the following we will denote by $C$ and $D$ the following subgroups of
$G\times \overset{p}{\cdots} \times G$:
\begin{align*}
C&= \{ (b^{i_1},\ldots,b^{i_p}) \mid i_1+\cdots+i_p=0 \},\\
D&=\{(g_1,\ldots,g_p)\in G'\times \cdots \times G' \mid g_1\cdots g_p\in \gamma_3(G)\}.
\end{align*}
Since $b$ is of infinite order, the subgroup~$C$ is a free abelian group of rank $p-1$ generated by the elements 
\begin{equation}
\label{ci}
c_i=(1,\ldots,1,b^{-1},b,1,\overset{i}{\ldots},1),\quad \text{ for } i\in\{0,1,\ldots,p-2\}.
\end{equation}
Note further that for all $i\in \{0,\ldots,p-2\}$ we have
\begin{equation}
\label{ci commutator}
c_i=\psi([b^{-1},a]^{b^{-i}}).
\end{equation}
Also since
\begin{equation}
\label{eqn:product all ci}
c_0 c_1 \cdots c_{p-2} = (b^{-1},1,\overset{p-2}{\ldots},1,b),
\end{equation}
it is clear that $C$ is normalised by $\sigma$ in the wreath product $W_p(G)$.

We start our study of commutators by identifying the images of $G'$, $\gamma_3(G)$ and $G''$
under $\psi$.

\begin{thm}
\label{thm:semidirect}
Let $G$ be a $p$-Basilica group, for $p$ a prime.
 The following hold:
\begin{enumerate}
\item
$\psi(G')= (G' \times \cdots \times G') \rtimes C$.
\item
$\psi(G'')=\gamma_3(G) \times \cdots \times \gamma_3(G)$.
\item
$\psi(\gamma_3(G)) = \langle y_0,\ldots,y_{p-2} \rangle \ltimes  D$,
where $y_{0}=c_0^{\,-p}([a,b^{-1}],1,\ldots,1)$ and $y_i=c_{i-1}c_i^{-1}$ for $1\le i\le p-2$. 
\end{enumerate}
\end{thm}

\begin{proof}
(i)
One inclusion is clear, taking into account that $G$ is weakly regular branch over $G'$, together with
(\ref{ci commutator}).
For the reverse inclusion, observe that $G'=\langle [b^{-1},a]\rangle^G$ implies
$\psi(G') \le \langle c_0 \rangle^{W_p(G)}$.
Now in the wreath product $W_p(G)$ the subgroup $G'\times \cdots \times G'$ is normal, the element~$\sigma$ normalises $C$ and $[C,G\times \cdots \times G]\le G'\times \cdots \times G'$.
Hence $(G'\times \cdots \times G')C\trianglelefteq W_p(G)$.
Since $c_0\in C$, the desired inclusion follows.

Finally, since $b$ has infinite order modulo $G'$ by Theorem \ref{thm:G/G'}(ii), we observe that the intersection of $G'\times \cdots \times G'$ and $C$ is trivial, and so their product is semidirect.

(ii)
We first show that $\gamma_3(G)\times \cdots \times \gamma_3(G)\le \psi(G'')$.
On the one hand, a straightforward computation shows that $[b,a,a]=1$, and then
$\gamma_3(G)=\langle [b,a,b]\rangle^G$.
Now since $G$ is weakly regular branch over $G'$, we have $(1,\ldots,1,[b,a])\in \psi(G')$, and on the other hand $(1,\ldots,1,b^{-1},b)=c_0\in \psi(G')$.
Hence $(1,\ldots,1,[b,a,b])\in \psi(G'')$ and the desired inclusion follows from Lemma
\ref{lem:getting-direct-product}.

To show the other inclusion, it is sufficient to prove that
$\psi(G')/(\gamma_3(G)\times \cdots \times \gamma_3(G))$ is abelian.
This is obvious from the expression for $\psi(G')$ obtained in (i).

(iii)
We have $\gamma_3(G)=[G',a][G',b]G''$ and consequently $\gamma_3(G)=[G',b]G''$
since $\psi(a)=(1,\ldots,1,b)$ clearly centralises $\psi(G')=(G'\times \cdots \times G')C$ modulo
$\psi(G'')=\gamma_3(G) \times \cdots \times \gamma_3(G)$.

Hence $\psi(\gamma_3(G))=[\psi(G'),\psi(b)](\gamma_3(G)\times \cdots \times \gamma_3(G))$.
Next consider the following subgroup of $\psi(\gamma_3(G))$,
\[
[G'\times \cdots \times G',\psi(b)] (\gamma_3(G)\times \cdots \times \gamma_3(G))
=
[G'\times \cdots \times G',\sigma] (\gamma_3(G)\times \cdots \times \gamma_3(G)),
\]
and observe that it corresponds to the commutator subgroup of the wreath product
$W_p(G'/\gamma_3(G))$, i.e.\ with the elements of the base group whose component-wise product is $1$ in $G'/\gamma_3(G)$.
Hence this subgroup coincides with $D$ and we have $\psi(\gamma_3(G))=[\psi(G'),\psi(b)]D$.

Now the factor group $\psi(G')/D$ is the direct product of the cyclic subgroups generated by $c_0,\ldots,c_{p-2}$ and by $z=([b,a],1,\ldots,1)$.
It is clear that $[c_i,\psi(b)]=y_i$ for $1\le i \le p-2$, that
\[
[c_0,\psi(b)] = c_{0}^{-1} (b,1,\ldots,1,b^{-1}) z = c_0^{-2} c_1^{-1}\cdots c_{p-3}^{-1} c_{p-2}^{-1} z,
\]
by using (\ref{eqn:product all ci}), and that $[z,\psi(b)]\in D$.
Hence
\[
\psi(\gamma_3(G))
=
\langle y_1,\ldots,y_{p-2},c_0^{-2} c_1^{-1}\cdots c_{p-3}^{-1} c_{p-2}^{-1}z \rangle D.
\]
Since
\[
c_0^{-2} c_1^{-1}\cdots c_{p-3}^{-1} c_{p-2}^{-1} y_{p-2}^{-1} y_{p-3}^{-2} \cdots y_1^{-(p-2)}
=
c_0^{\,-p},
\]
the result follows.
\end{proof}

\begin{cor}
Let $G$ be a $p$-Basilica group, for $p$ a prime.
Then the centraliser of~$G'$ in~$G$, and hence the centre of~$G$, is trivial. 
\end{cor}

\begin{proof}
Suppose that $g\in C_G(G')$ and let $\psi(g)=(g_1, \dots,g_p)\sigma^{k}$.
Since $g$ commutes with $[b^p,a]$ and $\psi([b^p,a])=(1, \dots, 1, [a,b])$, it follows that $p\mid k$.
Hence $g \in \St_G(1)$.
In view of Theorem~\ref{thm:semidirect}(i), all components $g_i$ also centralise $G'$ and consequently belong to $\St_G(1)$.
Repeating this process yields $g \in \St_G(n)$ for every $n\in\mathbb{N}$, and so $g=1$.
\end{proof}

For a group property~$\mathcal{P}$, recall that a group~$H$ is just non-$\mathcal{P}$ if $H$ does not have property~$\mathcal{P}$ but every proper quotient of~$H$ has~$\mathcal{P}$.

\begin{cor}
\label{cor: Just non solvable}
For a prime $p$, a $p$-Basilica group is just  non-solvable.
\end{cor}

\begin{proof} 
Using Theorem~\ref{thm:weakly-branch}, it suffices to show by~\cite[Lem.~10]{Zuk} that
$G'/\St_G(1)'$ is solvable.
Now from Theorem~\ref{thm:weakly-branch}(ii) and Theorem~\ref{thm:semidirect}(i) it is immediate that $G'/\St_G(1)'\cong C$ is abelian. 
\end{proof}

Next we give important information about the congruence quotient $G_n$: the orders of the images of $a$ and $b$, and the structure of its abelianisation.
In the remainder, let $\beta(n)=\lceil \nicefrac{n}{2} \rceil$.

\begin{thm}
\label{thm:n-order}
Let $G$ be a $p$-Basilica group, for a prime~$p$.
Then, for every $n\in\N$, we have:
\begin{enumerate}
\item
The orders of $a$ and $b$ modulo $\St_G(n)$ are $p^{\beta(n-1)}$ and $p^{\beta(n)}$, respectively.
\item
We have $|G_n:G_n'|=p^n$ and $G_n/G_n'\cong C_{p^{\beta(n-1)}} \times C_{p^{\beta(n)}}$, where
the first factor corresponds to $a$ and the second to $b$.
\end{enumerate}
\end{thm}

\begin{proof}
First of all, we prove by induction on $n\ge 0$ the following result: that if $n$ is even then
\begin{equation}
\label{eq:power of b}
b^{p^{\nicefrac{n}{2}}} \in \St_G(n)\backslash G'\St_G(n+1),
\end{equation}
and that if $n$ is odd then
\begin{equation}
\label{eq:power of a}
a^{p^{\nicefrac{(n-1)}{2}}} \in \St_G(n)\backslash G'\St_G(n+1).
\end{equation}
Note that this already implies (i), and furthermore that the orders of the images of $a$ and $b$ in
$G_n/G_n'$ are $p^{\beta(n-1)}$ and $p^{\beta(n)}$, respectively.

The result is obvious for $n=0$.
Now we suppose it holds for $n-1$.
If $n=2m+1$ is odd, we have $\psi(a^{p^m}) = (1,\ldots,1,b^{p^m})$, and since $b^{p^m}\in \St_G(n-1)$ by the induction hypothesis, we get $a^{p^m}\in \St_G(n)$.
Assume, by way of contradiction, that $a^{p^m}\in G'\St_G(n+1)$.
By applying $\psi$ and taking Theorem~\ref{thm:semidirect}(i) into account, we get
\[
(1,\ldots,1,b^{p^m}) \in (G'\St_G(n)\times \cdots \times G'\St_G(n)) C.
\]
Thus for some $i_1,\ldots,i_p$ summing up to $0$, we have
\[
(b^{i_1},\ldots,b^{i_{p-1}},b^{i_p+p^m}) \in G'\St_G(n)\times \cdots \times G'\St_G(n).
\]
By multiplying together all components, we get $b^{p^m}\in G'\St_G(n)$, contrary to the induction hypothesis.

Let us now consider the case when $n=2m$ is even.
Note that $\psi(b^{p^m}) = (a^{p^{m-1}},\ldots,a^{p^{m-1}})$.
Since $a^{p^{m-1}}\in\St_G(n-1)$, we get $b^{p^m}\in\St_G(n)$.
As before, suppose that $b^{p^m}\in G'\St_G(n+1)$.
Applying $\psi$ as above and looking at the first component, we get
\begin{equation}
\label{ed:bia^p^m in G'st(n)}
b^ia^{p^{m-1}}\in G'\St_G(n)
\end{equation}
for some integer $i$. 
Hence $b^{i}\in G'\St_G(n-1)$.
Since by induction $b^{p^{m-1}}\not\in G'\St_G(n-1)$, necessarily $p^m$ divides $i$.
Then by \eqref{ed:bia^p^m in G'st(n)} we obtain that $a^{p^{m-1}}\in G'\St_G(n)$, contrary to the induction hypothesis. 

We now prove (ii).
The abelian group $G_n/G_n'$ is generated by the images of $a$ and $b$, of orders
$p^{\beta(n-1)}$ and $p^{\beta(n)}$, respectively.
Hence $|G_n:G_n'|\le p^{\beta(n-1)+\beta(n)}=p^n$, and all assertions in (ii) immediately follow if we
show that $|G_n:G_n'|\ge p^n$.
To this purpose, observe that \eqref{eq:power of b} and \eqref{eq:power of a} imply that the difference $|G'\St_G(n): G'\St_G(n+1)|$ is non-trivial for every~$n$.
Thus
\[
|G_n:G_n'|=|G:G'\St_G(n)|=\prod_{i=0}^{n-1}|G'\St_G(i):G'\St_G(i+1)|\geq p^n,
\]
as desired.
\end{proof}

Now we can give a stronger form of Lemma \ref{lem:basilicaleveltransitive}.

\begin{thm}
\label{thm:super-strongly-fractal}
A $p$-Basilica group $G$,  for $p$ a prime, is super strongly fractal.
\end{thm}

\begin{proof}
Let $u_n=x_p\overset{n}{\dots}x_p$ for every $n\in\N$.
Since $G$ is level-transitive, it suffices to show that $\psi_{u_n}(\St_G(n))=G$ for all $n$.
Since
\[
\psi_{u_{2n-1}}(b^{p^n})=a, \quad \psi_{u_{2n-1}}(a^{p^{n-1}})=b, \quad
\psi_{u_{2n}}(b^{p^n})=b, \quad \text{and} \quad \psi_{u_{2n}}(a^{p^n})=a, 
\]
the result follows from Theorem~\ref{thm:n-order}.
\end{proof}

 The above result gives the first examples of weakly branch, but not branch, groups that are super strongly fractal; cf.~\cite[Prop.~3.11]{TUA} and \cite[Prop.~4.3]{Jone}.

For the following result, recall that a group $\mathcal{G}\le \text{Aut}(T)$ is \emph{saturated} if for any $n\in \mathbb{N}$ there exists a subgroup $H_n\le \text{St}_{\mathcal{G}}(n)$ that is characteristic in~$\mathcal{G}$ and $\psi_v(H_n)$ acts level-transitively on $T$ for all vertices~$v\in L_n$. If $\mathcal{G}$ is level-transitive and super strongly fractal, then it suffices to show this last property for a single vertex $u\in L_n$, since if we write $v=u^g$ with $g\in\mathcal{G}$ then by
(\ref{eqn:section conjugate}) we have
\[
\psi_v(H_n) = \psi_{u^g}(H_n^g) = \psi_u(H_n)^{g_u} = \psi_u(H_n),
\]
where the last equality follows from
$\psi_u(H_n)\trianglelefteq \psi_u(\St_{\mathcal{G}}(n))=\mathcal{G}$, since $\mathcal{G}$ is super strongly fractal.

\begin{cor}
\label{cor:Aut(G)}
Let $G$ be a $p$-Basilica group for a prime $p$. Then $\Aut(G)=N_{\Aut(T)}(G)$.
\end{cor}
 
\begin{proof} From~\cite[Prop.~7.5]{LN} it suffices to show that $G$ is saturated. We set $H_0=G$, $H_1=G'$, and $H_n=[H_{n-1},H_{n-2}^p]$ for all $n\ge 2$, which are characteristic subgroups of~$G$.
Note that $H_n\le \St_G(n)$ for every $n\in\mathbb{N}$, since the quotient of two consecutive level stabilisers in $G$ is an elementary abelian $p$-group.

Set $u_n=x_p\overset{n}{\dots}x_p$ for every $n\in\mathbb{N}$.
As explained above, we only need to show that $\psi_{u_n}(H_n)$ acts level-transitively on~$T$.
This will follow if we prove that $\psi_{u_{n-1}}(H_n)=G'$, since $\psi([b^{-1},a])=(1,\ldots,1,b^{-1},b)$, with $b$ acting transitively on the first level vertices, and since $G'\times\cdots\times G'\le \psi(G')$.

Let us then prove that $\psi_{u_{n-1}}(H_n)=G'$ for every $n\in\mathbb{N}$.
We use induction on $n$, the case $n=1$ being obvious.
Assume now that $n\ge 2$.
Clearly we only have to show that $G'\le \psi_{u_{n-1}}(H_n)$.
Observe first that $b\in\psi_{u_{n-1}}(H_{n-1})$ by the induction hypothesis.
Since also $b\in\psi_{u_{n-2}}(H_{n-2})$ (here we need to use that $H_0=G$ if $n=2$) and
$\psi(b^p)=(a,\ldots,a)$, we have $a\in\psi_{u_{n-1}}(H_{n-2}^p)$.
Consequently $[b,a]\in \psi_{u_{n-1}}([H_{n-1},H_{n-2}^p])=\psi_{u_{n-1}}(H_n)$.
Since $G'=\langle [b,a] \rangle^G$ and $\psi_{u_{n-1}}(H_n)\trianglelefteq G$, the proof is complete.
 \end{proof}

Examples of other groups acting on rooted trees with automorphism group equal to its normaliser in
$\Aut(T)$ are the Grigorchuk group and the Brunner-Sidki-Vieira group~\cite{LN}, and the branch multi-EGS groups~\cite{TUA}. 

Next we generalise Theorem \ref{thm:weakly-branch}(ii) and give a relation between rigid stabilisers and level stabilisers.

\begin{thm}
\label{thm:derived stabilisers}
Let $G$ be a $p$-Basilica group, for a prime $p$.
Then the following hold:
\begin{enumerate}
\item
$\psi_n(\St_G(n)')=G'\times\overset{p^n}{\cdots}\times G'$ for every $n\in\N$.
\item
$\Rist_G(n)=\St_G(n)'$ for every $n\ge 2$.
\end{enumerate}
\end{thm}

\begin{proof}
(i)
Since $G$ is self-similar, the inclusion $\subseteq$ is obvious.
For the reverse inclusion, a direct calculation shows that 
\begin{align}
\label{eq:psi[a,b]modd}
(1,\overset{p^n-1}\ldots,1,[b,a])=
\begin{cases}
\psi_n([b^{p^{\nicefrac{n}{2}}},a^{p^{\nicefrac{n}{2}}}]) & \text{ if $n$ is even},\\
\psi_n([a^{p^{\nicefrac{(n-1)}{2}}},b^{p^{\nicefrac{(n+1)}{2}}}]) & \text{ if $n$ is odd}.
\end{cases}
\end{align}
By Theorem~\ref{thm:n-order}, it follows that $(1,\overset{p^n-1}\ldots,1,[b,a])\in\psi_n(\St_G(n)')$.
Now the result follows from Lemma \ref{lem:getting-direct-product}.

(ii)
It suffices to show that $\psi_n(\Rist_G(n))=G'\times\overset{p^n}{\cdots}\times G'$ for every $n\ge 2$.
For $n=2$ we have
\begin{align*}
\psi(\Rist_G(2))
&\subseteq
(\Rist_G(1) \times \overset{p}{\cdots} \times \Rist_G(1)) \cap \psi(\Rist_G(1))
\\
&=
(A\times \overset{p}{\cdots} \times A)\cap (B\times \overset{p}{\cdots} \times B)
\\
&=
G'\times \overset{p}{\cdots} \times G',
\end{align*}
by using first Theorem \ref{thm:weakly-branch} and then Theorem \ref{thm:G/G'}(iii).
Hence
\begin{equation}
\label{eq:psi2(rist2)}
\psi_2(\Rist_G(2)) \subseteq \psi(G') \times \overset{p}{\cdots} \times \psi(G')
= (G'\times \overset{p}{\cdots} \times G')C \times \overset{p}{\cdots} \times
(G'\times \overset{p}{\cdots} \times G')C,
\end{equation}
by Theorem \ref{thm:semidirect}.
Now since $G$ is weakly regular branch over $G'$, we have
\[
G'\times \overset{p^2}{\cdots} \times G' \subseteq \psi_2(\Rist_G(2)).
\]
If this inclusion is strict then we may assume without loss of generality that some element
$(b^m,\ldots)$ with $m\ne 0$ belongs to $\psi_2(\Rist_G(2))$.
By the definition of rigid stabiliser, it follows that $(b^m,1,\ldots,1) \in\psi_2(\Rist_G(2))$,
which is contrary to (\ref{eq:psi2(rist2)}).
This proves the case $n=2$.
The general case follows in a similar fashion by induction on $n$.
\end{proof}

We close this section by determining the structure of the quotients $G'/\gamma_3(G)$, $G'/G''$, and
$\gamma_3(G)/G''$, which is key for Section~\ref{sec:CSP}.
We need a couple of lemmas.

\begin{lem}
\label{lem:psi stG'}
Let $G$ be a $p$-Basilica group, for a prime $p$. For every $n\in\mathbb{N}$, we have
\[
\psi(\St_{G'}(n))
=
\Big( \St_{G'}(n-1) \times \cdots \times \St_{G'}(n-1) \Big) \, C^{p^{\beta(n-1)}}.
\]
\end{lem}

\begin{proof}
The inclusion $\supseteq$ is obvious from Theorem \ref{thm:semidirect}(i) and
Theorem \ref{thm:n-order}(i).
For the other direction, let $g\in \St_{G'}(n)$ and write
$\psi(g)=(w_1,\ldots,w_p)(b^{i_1},\ldots,b^{i_p})$, where the first factor is in $G'\times \cdots \times G'$ and the second is in $C$.
Fix an index $j\in\{1,\ldots,p\}$.
Since $w_jb^{i_j}\in\St_G(n-1)$ we have $b^{i_j}\in G'\St_G(n-1)$, and then $p^{\beta(n-1)}$ divides $i_j$ by Theorem~\ref{thm:n-order}.
Thus $b^{i_j}\in \St_G(n-1)$ and it follows that also $w_j\in\St_G(n-1)$. This proves the result.
\end{proof}

\begin{lem}
\label{lem:order comm}
Let $G$ be a $p$-Basilica group, for $p$ a prime.
Then the order of $[a,b]$ modulo $\gamma_3(G)\St_{G'}(n)$ is at least $p^{\beta(n-1)}$ for every $n\in\N$.
\end{lem}

\begin{proof}
Since $[a,b]$ and $[a,b^{-1}]$ are inverse conjugate, we prove the result for the order of $[a,b^{-1}]$.
We use induction on $n$.
The result is obvious if $n=1$, so we suppose that $n\ge 2$.
If $[a,b^{-1}]^{p^m}\in \gamma_3(G)\St_{G'}(n)$, we want to show that $m\ge \beta(n-1)$.
By way of contradiction, we assume that $m<\beta(n-1)$.
By applying $\psi$ to $[a,b^{-1}]^{p^m}$ and using Theorem~\ref{thm:semidirect}(iii) and Lemma~\ref{lem:psi stG'}, we get
\begin{equation}
\label{decomposition}
c_{0}^{-p^m} = y d g c,
\end{equation}
where $y = y_0^{k_0} \ldots y_{p-2}^{k_{p-2}}$ for some $k_0,\ldots,k_{p-2}\in\mathbb{Z}$, $d\in D$,
$g\in \St_{G'}(n-1)\times \cdots \times \St_{G'}(n-1)$, and $c\in C^{p^{\beta(n-1)}}$.
If we reduce \eqref{decomposition} modulo $G'\times \cdots \times G'$ and use that $y_{0}$ reduces to $c_{0}^{-p}$, we get
\begin{equation}
\label{basis decomposition}
c_{0}^{\,p^m-pk_0+k_1}c_1^{-k_1+k_2}\cdots c_{p-3}^{-k_{p-3}+k_{p-2}}c_{p-2}^{-k_{p-2}}
\in
C^{p^{\beta(n-1)}}.
\end{equation}
Since $c_0,\ldots,c_{p-2}$ form a basis for the free abelian group~$C$, it follows that
all exponents in~\eqref{basis decomposition} are divisible by $p^{\beta(n-1)}$ and, as a consequence, so is $p^m-pk_0$.
Since $m<\beta(n-1)$, it follows that the $p$-part of $pk_0$ is~$p^m$.
Now since $B=\langle b \rangle G'$ is abelian modulo $\gamma_3(G)$, the map
\[
\begin{matrix}
\tau & \colon & B \times \cdots \times B & \longrightarrow & B/\gamma_3(G)\hfill
\\[5pt]
& & (g_1,\ldots,g_p) & \longmapsto & g_1\cdots g_p \gamma_3(G)
\end{matrix}
\]
is a group homomorphism.
Observe that both $C$ and $D$ lie in the kernel of~$\tau$, and that
$\tau(y_0)=[a,b^{-1}]\gamma_3(G)$.
Hence by applying $\tau$ to \eqref{decomposition}, we get
\[
[a,b^{-1}]^{k_0}\in \gamma_3(G) \St_{G'}(n-1).
\]
Since the $p$-part of $k_{0}$ is $p^{m-1}$, by the induction hypothesis we have $m-1\ge\beta(n-2)$,
and so $m\ge \beta(n-2)+1\ge \beta(n-1)$, contrary to our assumption.
This completes the proof.
\end{proof}

\begin{thm}
\label{thm:key}
Let $G$ be a $p$-Basilica group, for $p$ a prime. Then:
\begin{enumerate}
\item
$G'/\gamma_3(G)\cong \mathbb{Z}$.
\item
$G'/G''\cong \mathbb{Z}^{2p-1}$.
\item
$\gamma_3(G)/G''\cong \mathbb{Z}^{2p-2}$.
\end{enumerate}
\end{thm}

\begin{proof} 
(i) We observe that $G'/\gamma_3(G)$ is cyclic and generated by the image of $[a,b]$.
From Lemma~\ref{lem:order comm}, the order of $[a,b]$ tends to infinity modulo
$\gamma_3(G)\St_{G'}(n)$ as $n$ goes to infinity.
Hence the statement immediately follows. 

(ii) The result follows from (i) and from Theorem~\ref{thm:semidirect} since we have
\begin{align*}
\frac{G'}{G''}
\cong
\frac{\psi(G')}{\psi(G'')}
&\cong
\left( \frac{G'}{\gamma_3(G)} \times \overset{p}{\cdots} \times \frac{G'}{\gamma_3(G)} \right)
\times C
\\
&\cong
\mathbb{Z} \times \overset{p}{\cdots} \times \mathbb{Z}\times \mathbb{Z}^{p-1}.
\end{align*}
This completes the proof since (iii) is a straightforward consequence of (i) and (ii).
\end{proof}

\begin{cor}\label{cor:Heisenberg}
Let $G$ be a $p$-Basilica group, for $p$ a prime. Then $G/\gamma_3(G)$ is isomorphic to the integral Heisenberg group~$H_3(\mathbb{Z})$.
\end{cor}

\begin{proof}
The proof is analogous to~\cite[Prop.~23]{pcongruence} or to~\cite[Prop.~4.8]{Francoeur-paper}, where it was proved for the case $p=2$, using different approaches.
\end{proof}

As noted in~\cite[Cor.~4.9]{Francoeur-paper},  the previous result yields an alternative proof that the $p$-Basilica groups are not branch. As every proper quotient of a branch group is virtually abelian and the integral Heisenberg group is not virtually abelian, the result follows.

\section{Congruence subgroup properties and Hausdorff dimension}\label{sec:CSP}

\subsection{Congruence subgroup properties}
Let $G$ be a $p$-Basilica group, for $p$ a prime.
Since $G/G'\cong \mathbb{Z}\times \mathbb{Z}$, the group~$G$ does not have the congruence subgroup property as  all quotients of $G$ by level stabilisers are $p$-groups.
In this subsection we show that $G$ has the $p$-congruence subgroup property ($p$-CSP for short) but not the weak congruence subgroup property.

We recall that if $\mathcal{G}$ is a subgroup of $\Aut(T)$ and $N \trianglelefteq \mathcal{G}$, we say that
$\mathcal{G}/N$ has the $p$- \textit{congruence subgroup property} (or that $\mathcal{G}$ has the $p$-\textit{congruence subgroup property modulo} $N$)  if every
$J \trianglelefteq \mathcal{G}$, such that $\mathcal{G}/J$ is a finite $p$-group and $N\leq J$, contains some level stabiliser of~$\mathcal{G}$.
According to \cite[Lem.~6]{pcongruence}, if $N\le M$ are two normal subgroups of $\mathcal{G}$
such that both $\mathcal{G}/M$ and $M/N$ have the $p$-CSP then also $\mathcal{G}/N$ has the
$p$-CSP.

We need a couple of lemmas before proving that the $p$-Basilica groups have the $p$-CSP.

\begin{lem}
\label{lem:p-CSP A/B}
Let $N$ and $\mathcal{G}$ be subgroups of $\Aut(T)$ with $N\trianglelefteq \mathcal{G}$ and
$\mathcal{G}/N$ free abelian of rank~$r$, for some $r\in\mathbb{N}$.
Suppose that, for large enough $n\in\mathbb{N}$, we have
\begin{equation}
\label{A/B}
\mathcal{G}/N\St_{\mathcal{G}}(n)
\cong
C_{p^{\lambda_1(n)}} \times \cdots \times C_{p^{\lambda_r(n)}},
\end{equation}
with $\lim_{n\to\infty} \lambda_i(n)=\infty$ for $1\le i\le r$.
Then $\mathcal{G}/N$ has the $p$-CSP.
\end{lem}

\begin{proof}
Let $N\le J\trianglelefteq \mathcal{G}$, where $|\mathcal{G}:J|=p^m$, for some $m\in\mathbb{N}$.
Then $\mathcal{G}^{p^m}\le J$.
Now choose an integer $n$ such that $\lambda_i(n)\ge m$ for $1\le i\le r$.
By~\eqref{A/B}, for large enough $n$ we have
\[
|\mathcal{G}/N\St_{\mathcal{G}}(n):(\mathcal{G}/N\St_\mathcal{G}(n))^{p^m}| = p^{rm}
= |\mathcal{G}/N:(\mathcal{G}/N)^{p^m}|,
\]
which implies $N\St_{\mathcal{G}}(n)\mathcal{G}^{p^m}=N\mathcal{G}^{p^m}$.
Hence $\St_\mathcal{G}(n)\le N\mathcal{G}^{p^m}\le J$, and
$\mathcal{G}/N$ has the $p$-CSP.
\end{proof}

\begin{lem}
\label{lem:p-CSP weakly}
Let $\mathcal{G}$ be a subgroup of $\Aut(T)$ that is weakly regular branch over a normal subgroup $K$.
Let $N$ be a normal subgroup of $\mathcal{G}$ such that:
\begin{enumerate}
\item
$K'\le N\le K$.
\item
If $L=\psi^{-1}(N\times \cdots \times N)$ then $G/N$, $N/L$, and $N/K'$ have the $p$-CSP.
\end{enumerate}
Then $\mathcal{G}$ has the $p$-CSP.
\end{lem}

The last lemma is a slight generalisation of~\cite[Thm.~1]{pcongruence}, which corresponds to the case when $N$ is chosen so that $L\le K'$, and its proof is very similar.
We leave the details to the reader.

\begin{thm}
Let $G$ be a $p$-Basilica group, for a prime~$p$. Then $G$ has the $p$-CSP.
\end{thm}

\begin{proof}
We apply Lemma~\ref{lem:p-CSP weakly} with $K=N=G'$.
Thus if $L=\psi^{-1}(G'\times \cdots \times G')$ then it suffices to prove that $G/G'$, $G'/L$ and $L/G''$ have the $p$-CSP.
Note that then also $G'/G''$ has the $p$-CSP.

First of all, the factor group~$G/G'$ has the $p$-CSP by Lemma~\ref{lem:p-CSP A/B}, since
$G/G'\St_G(n)\cong C_{p^{\beta(n-1)}}\times C_{p^{\beta(n)}}$ with 
$\beta(n)=\lceil n/2 \rceil$, according to Theorem~\ref{thm:n-order}(ii).

Next we deal with $G'/L$, which is free abelian of rank~$p-1$ by Theorem~\ref{thm:semidirect}(i).
By Lemma~\ref{lem:psi stG'}, we have
$\psi(L\St_{G'}(n))=(G'\times \cdots \times G')C^{p^{\beta(n-1)}}$ and consequently
$G'/L\St_{G'}(n)\cong C/C^{p^{\beta(n-1)}}$.
Hence this case also follows from Lemma~\ref{lem:p-CSP A/B}.

Let us finally consider the case of $L/G''$.
We have 
\[
L/G''\cong (G'\times \cdots \times G')/(\gamma_3(G)\times \cdots \times \gamma_3(G))
\cong \mathbb{Z}^p.
\]
Since
\[
\psi(G''\St_L(n))=\gamma_3(G)\St_{G'}(n-1)\times \cdots \times \gamma_3(G)\St_{G'}(n-1),
\]
it follows that
\[
L/G''\St_L(n) \cong G'/(\gamma_3(G)\St_{G'}(n-1))\times \cdots \times G'/(\gamma_3(G)\St_{G'}(n-1)),
\]
and we can once again apply Lemma~\ref{lem:p-CSP A/B}, by taking into account
Lemma~\ref{lem:order comm}.
\end{proof}

We note that in~\cite[Thm.~1.10]{PR} it was shown that $s$-generator Basilica groups, for $s>2$, have the $p$-CSP. It is worth mentioning that there are key structural differences between these groups and the $p$-Basilica groups; compare \cite[Thm.~1.9]{PR}.

Now we complete the proof of Theorem~\ref{thm:CSP}. 

\begin{thm}
Let $G$ be a $p$-Basilica group, for a prime~$p$. Then $G$ does not have the weak congruence subgroup property.
\end{thm}

\begin{proof}
Let $q\ne p$ be a prime, and let $N=\langle a^q,b^q,[a,b]^q\rangle \gamma_3(G)$, which is normal and of finite index in $G$.
By Corollary~\ref{cor:Heisenberg}, we have $G/N\cong H_3(q)$.
We claim that $\St_G(n)'\not\le N$ for every odd $n$, which is enough to prove the theorem.
Arguing by way of contradiction, since by Theorem~\ref{thm:derived stabilisers} we have
$\psi_n(\St_G(n)')=G'\times\overset{p^n}\cdots\times G'$, and according to \eqref{eq:psi[a,b]modd}
\[
\psi_n([a^{p^{\nicefrac{(n-1)}{2}}},b^{p^{\nicefrac{(n+1)}{2}}}])
=
(1,\overset{p^n-1}\ldots,1,[b,a]) \in G'\times\overset{p^n}\cdots\times G'
\]
for odd $n$, it follows that $[a^{p^{\nicefrac{(n-1)}{2}}},b^{p^{\nicefrac{(n+1)}{2}}}]\in N$.
As $\gamma_3(G)\le N$, we get $[a,b]^{p^n}\in N$.
Since also $[a,b]^q\in N$, we conclude that $[a,b]\in N$.
This contradicts the fact that $G/N\cong H_3(q)$.
\end{proof}

\subsection{Hausdorff dimension}\label{sec:Hausdorff}

In this subsection we determine the orders of the congruence quotients of the $p$-Basilica groups, and we compute their Hausdorff dimensions.

\begin{proof}[Proof of Theorem~\ref{thm:hdim}]
(i) We argue by induction on $n$.
The case $n=1$ is clear, so we assume $n\ge 2$.
Write $n=2m+e$, with $e=0$ or $1$. We need to establish that
\[
\log_p |G:\St_G(n)|
=
\frac{p^{n+1}-p^{1+e}}{p^2-1} + m+e.
\]
Note that, by Theorem~\ref{thm:n-order},
\[
|G:\St_G(n)|
=
|G:G'\St_G(n)| \, |G'\St_G(n):\St_G(n)|
=
p^n \, |G':\St_{G'}(n)|
\]
and that $|G':\St_{G'}(n)|$ coincides with
\begin{align*}
|\psi(G'):\psi(\St_{G'}(n))|
&=
|(G'\times \overset{p}{\cdots} \times G')C:
(\St_{G'}(n-1)\times \overset{p}{\cdots} \times \St_{G'}(n-1))C^{p^{\beta(n-1)}}|
\\
&=
p^{(p-1)\beta(n-1)} \, |G':\St_{G'}(n-1)|^p
\\
&=
p^{(p-1)\beta(n-1)-p(n-1)} \, |G:\St_G(n-1)|^p,
\end{align*}
where we have used Lemma \ref{lem:psi stG'} and the fact that $C$ is free abelian of rank $p-1$.
Here $\beta(n-1)=\lceil \nicefrac{(n-1)}{2}\rceil$ as before.
Thus
\begin{align*}
\log_p |G:\St_G(n)|
&=
p \log_p |G:\St_G(n-1)| + (p-1)(\beta(n-1)-n+1)+1
\\
&=
p \log_p |G:\St_G(n-1)| - (p-1)(m+e-1)+1,
\end{align*}
since $\beta(n-1)=m$.
Now the result follows from the induction hypothesis.

(ii) In order to get the Hausdorff dimension of $\overline G$ in $\Gamma$, we just need to take into account formula (\ref{eqn:hausdorff dim}) and the fact that
\[
\log_p |\Gamma:\St_{\Gamma}(n)|
=
1+p+\cdots+p^{n-1}
=
\frac{p^n-1}{p-1}.\qedhere
\]
\end{proof}

We remark that the Hausdorff dimension of the Basilica group was given by Bartholdi in~\cite{Bartholdi}.
Also the Hausdorff dimensions of the generalised Basilica groups were computed
in~\cite[Thm.~1.7]{PR} by using a very different approach.


\section{Further properties}\label{sec:growth}

\subsection{Growth and amenability}

Before proving the main results of this subsection, we need some preliminary definitions, namely the notions of growth of groups and amenability.

Let~$\mathcal{G}$ be a group generated by a finite symmetric subset~$S$. 
The length function on $\mathcal{G}$ is a metric on~$\mathcal{G}$ and therefore one can define the ball of radius $n$:
\[
B(n) = \{g \in \mathcal{G}:|g|\leq n\}.
\]
We say that the map $\gamma: \mathbb{N}_{0} \longrightarrow [0, \infty)$ where $\gamma(n) = |B(n)|$, is the \textit{growth function} of~$\mathcal{G}$.

If we consider two growth functions $\gamma_{1}, \gamma_{2}$, we say that $\gamma_{2}$ \textit{dominates} $\gamma_{1}$ and we write $\gamma_{1} \preceq \gamma_{2}$ if there exist $C, \alpha > 0$ such that $\gamma_{1}(n) \leq C\gamma_{2}(\alpha n)$ for every $n \in \mathbb{N}$. If $\gamma_{1} \preceq \gamma_{2}$ and $\gamma_{2} \preceq \gamma_{1}$, we write $\gamma_{1} \sim \gamma_{2}$. It is easy to see that this is an equivalence relation. Notice also that all growth functions of a finitely generated group are equivalent. 

If $\gamma(n) \preceq n^{a}$ for some~$a\in\mathbb{N}$, we say that~$\mathcal{G}$ has \textit{polynomial growth}. Instead~$\mathcal{G}$ is said to have \textit{exponential growth} if 
$\lim_{n\to\infty} \gamma(n)^{1/n} > 1$
(notice that such a limit always exists). Finally  $\gamma(n)$ has \textit{intermediate growth} if $\gamma(n)$ is equivalent to neither of the above. Notice that it is also common to say that a group~$\mathcal{G}$ has \textit{subexponential growth} if 
$\lim_{n\to\infty} \gamma(n)^{1/n} = 1$, or equivalently,
if $\gamma(n) = e^{f(n)}$ for some (increasing) function $f: \mathbb{N} \longrightarrow \mathbb{R}^{+}$ satisfying $\lim_{n\to \infty} f(n)/n = 0$.

\smallskip

Next, we say that a group~$\mathcal{G}$ is \textit{amenable} if there is a finitely additive left-invariant
measure~$\mu$ on the subsets of $\mathcal{G}$ such that $\mu(\mathcal{G}) = 1$.
We denote the class of amenable groups by~$AG$.
The class~$EG$ of \textit{elementary amenable} groups is the smallest class of groups containing all abelian groups and finite groups and closed under quotients, subgroups, extensions and direct unions.
We have $EG \subseteq AG$, and this inclusion is strict.
Furthermore the class~$SG$ of \textit{elementary subexponentially amenable} groups is the smallest class of groups which contains all groups of subexponential growth and is closed under taking subgroups, quotients, extensions, and direct unions.
Of course, the class~$SG$ contains the class~$EG$.

In the following we determine the growth of a $p$-Basilica group $G$, for $p$ an odd prime, and we prove that $G$ is amenable but not elementary subexponentially amenable. The corresponding versions of Theorem~\ref{thm:exponentialgrowth} and Lemma~\ref{lem:amenable} for $p=2$ were proved in~\cite[Lem.~4 and Prop.~4]{Zuk} and in~\cite[Cor.~9]{Juschenko}, respectively.

\begin{thm}\label{thm:exponentialgrowth}
Let $G=\langle a,b\rangle$ be a $p$-Basilica group, for $p$ an odd prime. Then the semigroup generated by $a$ and $b$ is free. Consequently, the group $G$ is of exponential  growth.
\end{thm}

\begin{proof}
Let $u$ and $v$ be two different words representing the same element in the semigroup generated by $a$ and $b$, and with $\rho=\max(|u|,|v|)$ minimal. We note that $|u|_b\equiv_p |v|_b$, where $|u|_b$ denotes the $b$-length in $u$, which is equivalent to the number of occurrences of~$b$ in~$u$. A direct check shows that $\rho\ge 4$.

Suppose first that $u$ contains no $b$'s. Therefore we have $u=a^{i}=\psi^{-1}((1,\ldots,1,b^{i}))$, for some $i\in\mathbb{N}$. Since $|v|_b$ is a non-zero multiple of~$p$, one deduces that $v_1$ is a non-empty word, where $\psi(v)=(v_1,\ldots,v_p)$. Indeed, every component of $\psi(v)$ contains an~$a$. Certainly $|v_1|<\rho$, and this contradicts the minimality of~$\rho$. So the number of occurrences of~$b$ in~$u$ is at least one.

Suppose that both $|u|_b\equiv_p |v|_b\equiv_p 1$. Apart from the possibilities $a^mb$, for $m\in\mathbb{N}$, all sections of $u$ and $v$ will have length strictly less than~$|u|$ and~$|v|$ respectively. In order to not contradict the minimality of~$\rho$, we must have $u=a^{m_1}b=\psi^{-1}((1,\ldots,1,b^{m_1}a)\sigma)$ and $v=a^{m_2}b=\psi^{-1}((1,\ldots,1,b^{m_2}a)\sigma)$ for some $m_1,m_2\in\mathbb{N}$. However, as noted above, all sections of $b^{m_1}a$ and $b^{m_2}a$ decrease in length. Hence $|u|_b\equiv_p|v|_b\equiv_p k$ for $k>1$. Now in this case, all sections of $u$ and $v$ have length strictly smaller than $u$ and $v$ respectively. This again contradicts the minimality of~$\rho$, and the proof is complete.
\end{proof}

\begin{lem}\label{lem:amenable}
For $p$ a prime, the $p$-Basilica group $G$ is amenable but not elementary subexponentially amenable. In particular it is not elementary amenable.
\end{lem}

\begin{proof}
For $p=2$, the result follows from~\cite[Prop.~13]{Zuk} and~\cite{BV05}. Hence we assume that $p$ is odd.
From~\cite{BKN} it follows that $G$ is amenable, so it suffices  to show that $G$ is not elementary subexponentially amenable. Since $G$ is weakly regular branch over $G'$, from \cite[Cor.~3]{Juschenko} the result follows provided that
$\psi_u(\St_{G'}(u))$ contains $G$ for some vertex $u$. We observe that
\[
    \psi([b^{-1},a]^p)=(1,\overset{p-2}{\ldots},1,b^{-p},b^p)\quad\text{and}\quad \psi([a,b^p])=(1,\overset{p-1}{\ldots},1,[b,a]),
\]
thus $\psi_{u}([b^{-1},a]^p)=a$ and $\psi_{u}([a,b^p])=b$, where $u=x_px_p$.
This completes the proof.
\end{proof}

\subsection{An $L$-presentation}

For an alphabet~$S$, we denote by $F_S$ the free group on~$S$. A group~$\mathcal{G}$ has an \emph{$L$-presentation}, also called \emph{endomorphic presentation}, if there exists an alphabet~$S$, sets $Q$ and $R$ of reduced words in~$F_S$, and a set~$\Phi$ of group homomorphisms $\phi:F_S \rightarrow F_S$ such that $\mathcal{G}$ is isomorphic to a group with the following presentation
\[
\big\langle S \mid  Q\cup \bigcup_{\phi\in \Phi^*} \phi(R) \big\rangle,
\]
where $\Phi^*$ is the monoid generated by~$\Phi$; that is, the closure of $\{1\}\cup \Phi$ under composition.

An $L$-presentation is \emph{finite} if $S$, $Q$, $R$ are finite and $\Phi=\{\phi\}$ consists of just one element.
Finite $L$-presentations have been computed for the first Grigorchuk group~\cite{Lysionok}, the Brunner-Sidki-Vieira group~\cite{BSV}, the Grigorchuk supergroup~\cite{BG99}, the Fabrykowski-Gupta group~\cite{endomorphic}, the Gupta-Sidki 3-group~\cite{endomorphic}, and the twisted twin of the Grigorchuk group~\cite{BS}.

An $L$-presentation for the Basilica group is given in~\cite{Zuk}, and this $L$-presentation is not finite, since the corresponding set~$\Phi$ consists of more than one element. This is also the case for the $L$-presentations of the $p$-generator Basilica groups acting on the $p$-adic tree and generalised Basilica groups; see~\cite{Sasse} and \cite{PR} respectively.

Following the strategy in~\cite[Sec.~4]{Zuk}, one can easily check that the $p$-Basilica groups have the following $L$-presentation.

\begin{thm}\label{pr:11}
Let $G$ be a $p$-Basilica group, for a prime~$p$. The group~$G$ has the presentation
\[
G=\big\langle a,b\mid \xi^k\big(\theta^m([a,a^{b^l}])\big)=1 \text{ for }k,m\in\mathbb{N}\cup\{0\} \text{ and }l\in\{1,\ldots,p-1\}\big\rangle,
\]
where 
\begin{align*}
    \xi:& \,a\mapsto b^p\qquad\text{and}\qquad\theta:a\mapsto a^{b^p+1}\\
    &\,b\mapsto a \,\,\quad\qquad\qquad\qquad b\mapsto b
\end{align*}
are endomorphisms of~$F_{\{a,b\}}$.
\end{thm}

\smallskip


\subsection{Virtually nilpotent quotients and maximal subgroups}
\label{sec:nilpotent+maximal}

In this final subsection we study nilpotency and virtual nilpotency of quotients of a $p$-Basilica group $G$, and we prove Theorem~\ref{thm:maximal} about maximal subgroups of $G$.
The following lemma will be useful for both purposes.

\begin{lem}
\label{lem:wreath quotient}
Let $G$ be a $p$-Basilica group, for a prime $p$.
Then $G$ has a proper quotient isomorphic to $W_p(\mathbb{Z})$.
\end{lem}

\begin{proof}
Let $L=\psi^{-1}(G'\times \cdots \times G')$.
We have $G=A\langle b \rangle$, and on the other hand $\psi(A)=B\times \cdots \times B$ by
Theorem~\ref{thm:weakly-branch}(i).
Hence $\psi$ induces an isomorphism between $G/L$ and the semidirect product
$(B/G'\times \cdots \times B/G')\ltimes \langle \psi(b) \rangle$.
Observe that $\psi(b)=(1,\ldots,1,a)\sigma$ acts as $\sigma$ on the direct product of $p$ copies of
$B/G'$, and that $\psi(b^p)=(a,\ldots,a)$ acts trivially.
If we set $N=L\langle b^p \rangle$ then it is clear that $N\trianglelefteq G$ and that
$G/N\cong W_p(\mathbb{Z})$, since $B/G'\cong \mathbb{Z}$ by Theorem~\ref{thm:G/G'}(ii).
\end{proof}

Recall from Corollary~\ref{cor: Just non solvable} that the $p$-Basilica groups are just non-solvable. In~\cite[Sec.~8.3]{Francoeur} it was shown that the Basilica group is not just non-nilpotent.
On the other hand, by~\cite[Lem.~8.3.5 and Prop.~8.3.6]{Francoeur}, all proper quotients of the Basilica group are virtually nilpotent.
We extend these results to the $p$-Basilica groups for all primes~$p$.

\begin{thm}
\label{thm:nilpotency}
Let $G$ be a $p$-Basilica group, for a prime $p$.
Then:
\begin{enumerate}
\item
The group~$G$ is not just non-nilpotent.
\item
Every proper quotient of~$G$ is virtually nilpotent, but $G$ itself is not virtually nilpotent.
\end{enumerate}
\end{thm}

\begin{proof}
(i)
By Lemma~\ref{lem:wreath quotient}, the group $G$ has a proper quotient isomorphic to $W_p(\mathbb{Z})$.
By the main result in \cite{Baumslag}, this wreath product is not nilpotent.
Hence $G$ is not just non-nilpotent.

(ii)
From Theorem~\ref{thm:semidirect}(ii), the map $\psi$ induces an embedding of $G/G''$
into the wreath product $W_p(G/\gamma_3(G))$.
Since the latter is virtually nilpotent, also is $G/G''$.

Now since $G$ is weakly regular branch over $G'$ and $G/G''$ is virtually nilpotent, it follows that every proper quotient of $G$ is also virtually nilpotent by~\cite[Thm.~4.10]{Francoeur-paper}.
On the other hand, the group $G$ is not virtually nilpotent by Gromov's celebrated theorem~\cite{Gromov}, in light of Theorem~\ref{thm:exponentialgrowth}.
\end{proof}

Let us now consider the maximal subgroups of $G$.
We first prove that $G$ does not possess maximal subgroups of infinite index.
The proof  is analogous to that of~\cite[Sec.~4.4]{Francoeur-paper}, however with a necessary change to the end of~\cite[Prop.~4.27]{Francoeur-paper}. Due to the proof being so similar, we refer the reader to~\cite[Sec.~4.4]{Francoeur-paper}, and only record here the part that needs to be changed.

Recall that a subgroup $H$ of a group $\mathcal{G}$ is \emph{prodense} if $HN=\mathcal{G}$ for all non-trivial normal subgroups
$N$ of $\mathcal G$.
Since a maximal subgroup of infinite index is a proper prodense subgroup, it suffices to show that there are no proper prodense subgroups in a $p$-Basilica group~$G$. For $H$ a proper prodense subgroup of~$G$, by \cite[Lem.~3.1 and Thm.~3.2]{Francoeur-paper}, for all vertices $u\in T$, the subgroup $\psi_u(\St_H(u))$ is a proper prodense subgroup of~$G$. We consider a prodense subgroup~$H$ of~$G$, and seek a vertex $u$ such that $\psi_u(\St_H(u))=G$, which then proves the theorem. 

As in~\cite[Prop.~4.27]{Francoeur-paper}, there is a vertex $v$ such that either $ab,b^{-1}a\in \psi_v(\St_H(v))$ or $ba,b^{-1}a\in \psi_v(\St_H(v))$. 
In the former case, we obtain $a^2\in \psi_v(\St_H(v))$. Since $p$ is an odd prime, it follows that $b^{-1}a^p\in \psi_v(\St_H(v))$. Now
$\psi((b^{-1}a^p)^p)=(a^{-1}b^p,\ldots,a^{-1}b^p)$ and 
$\psi(a^{-1}b^p)=(a,\ldots,a,b^{-1}a)$.
Therefore, for $u=vx_1x_1$, we have either $a,ab\in \psi_u(\St_H(u))$ or $a,ba\in \psi_u(\St_H(u))$, and we are done.
In the latter case, we have $ba, b^{-1}a\in \psi_v(\St_H(v))$, and so $b^{2}\in \psi_v(\St_H(v))$. As before, we obtain $b^pa=\psi^{-1}((a,\ldots,a,ab))\in \psi_v(\St_H(v))$. Setting $u=vx_1$, we see that $a,ba\in \psi_u(\St_H(u))$ and the result follows.

We conclude by showing the existence of non-normal maximal subgroups in the $p$-Basilica groups.

\begin{pr}
Let $G$ be a $p$-Basilica group, for $p$ an odd prime.
Then for every prime $q$ such that $p$ divides $q-1$, the group~$G$ has a non-normal subgroup of index $q$.
\end{pr}

\begin{proof}
By Lemma~\ref{lem:wreath quotient}, the group~$G$ has a quotient isomorphic to $W_p(\ZZ)$, and so also a quotient isomorphic to $W_p(\ZZ/q\ZZ)$.
Thus it suffices to find a non-normal subgroup of index $q$ in the latter group.

Let $V=\ZZ/q\ZZ\times \cdots \times \ZZ/q\ZZ$ be the base group of $W_p(\ZZ/q\ZZ)$.
The characteristic polynomial corresponding to the action of $\sigma$ on $V$ is $X^p-1$, which by the condition that $p$ divides $q-1$, has $p$ different roots in $\ZZ/q\ZZ$.
Let $\lambda\ne 1$ be one of these roots, and let $U=\langle u \rangle$ be the eigenspace of
$\lambda$ in $V$.
Then we can write $V=U\times K$ for a suitable subgroup $K$.
If we set $H=K\langle \sigma \rangle$ then $H$ has index $q$ in $W_p(\ZZ/q\ZZ)$.
At the same time, $H$ is not a normal subgroup of $W_p(\ZZ/q\ZZ)$, since otherwise
$[u,\sigma]=u^{\lambda-1}\ne 1$ belongs to $U\cap H=1$.
\end{proof}

Observe that there are actually infinitely many non-normal maximal subgroups in a $p$-Basilica group, due to Dirichlet's theorem about primes in arithmetic progressions.



\end{document}